\documentclass[12pt, article]{amsart}
\usepackage{amsmath, amsthm, amscd, amsfonts, amssymb, graphicx, color}
\usepackage{mathrsfs}

\newtheorem{theorem}{Theorem}[section]
\newtheorem{lemma}[theorem]{Lemma}
\newtheorem{proposition}[theorem]{Proposition}
\newtheorem{corollary}[theorem]{Corollary}

\theoremstyle{definition}

\newtheorem{example}[theorem]{Example}

\theoremstyle{remark}
\newtheorem{remark}[theorem]{Remark}

\usepackage[all]{xy}
\numberwithin{equation}{section}

\begin{document}
\title[Numerical radius and norm bounds ]{Inner product bounds via the Moore-Penrose inverse with applications}
\author[A.
Sheikhhosseini ]{Alemeh
Sheikhhosseini}
\address{Department of Pure Mathematics, Faculty of Mathematics and Computer, Shahid Bahonar University of Kerman, Kerman, Iran} \email{sheikhhosseini@uk.ac.ir}

 \author[Gh. Aghamollaei]{Gholamreza Aghamollaei}
\address{Department of Pure Mathematics, Faculty of Mathematics and Computer, Shahid Bahonar University of Kerman, Kerman, Iran \newline
Mahani Math Center, Afzalipour Research Institute, Shahid Bahonar University of Kerman, Kerman, Iran \newline Corresponding author.}
\email{aghamollaei@uk.ac.ir}

 \author[M. Sababheh]{Mohammad Sababheh}
\address{Department of Basic Sciences, Princess Sumaya University for Technology, Amman 11941, Jordan} \email{sababheh@psut.edu.jo, sababheh@yahoo.com}


\maketitle
\begin{abstract}
In this paper, by implementing the generalized inverse, known as the Moore-Penrose inverse, for Hilbert space operators, several inner product bounds that are of Cauchy-Schwarz type are presented. As a standard application, some new refined and generalized versions of numerical radius-norm bounds are also stated, in a generalized form that involves different combinations of operators.
\end{abstract}

 \section{Introduction}

Let $ \mathcal{B}(\mathcal{H}) $ denote the $ \mathcal{C^{*}}- $algebra of all bounded linear operators from a complex Hilbert space $ \mathcal{H} $ to itself, with identity $I$. The numerical range of an operator $ T \in \mathcal{B}(\mathcal{H}) $ is defined to be the complex set
\[W(T)=\{ \langle Tx,x\rangle ~:~x\in \mathcal{H},~\|x\|=1\},\]
where $ \langle\cdot, \cdot\rangle $ is the inner product defined on $ \mathcal{H},$ and $ \Vert\cdot\Vert $ is the induced norm on $\mathcal{H}. $ It is well known that $ W(T) $ is a bounded convex set, but the exact shape of it is uncertain; see \cite{gust}.

 Related to the numerical range, recall that the numerical radius $\omega(T)$ of $T\in\mathcal{B}(\mathcal{H})$ is defined by
\[\omega(T)=\sup\{ |\langle Tx,x\rangle |~:~x\in \mathcal{H},~\|x\|=1\}.\]
It is well known that $\omega (\cdot)$ defines a norm on $\mathcal{B}(\mathcal{H})$, which is equivalent to the usual operator norm $\|\cdot\|$, defined by $\|T\|=\sup_{\|x\|=1}\|Tx\|$. In fact, for any $ T \in \mathcal{B}(\mathcal{H})$, one has \cite[Theorem 1.3-1]{gust}
\begin{equation}\label{eq1}
\dfrac{1}{2} \|T\| \leq \omega(T) \leq \|T\|.
\end{equation}
The first inequality in \eqref{eq1} becomes an equality if $T^2=O,$ where $O$ is the zero operator in $\mathcal{B}(\mathcal{H})$, while the second inequality becomes an equality if $T$ is normal in the sense that $TT^*=T^*T,$ where $T^*$ is the adjoint of $T$.

 It has been of adequate interest to find further relations between the numerical radius and operator norm. To be brief, we refer the reader to the list \cite{bhu1, S.S, hadkitt, kit_frob, Kittaneh 2, kitt_mosl, sab_num_laa, sab_num_lama, Sababheh, sahoo, satt_mosl_shebr, Sattari, sheyb, Z-M, zam_2021} as a reading list on this topic. Of recent interest, the role of the Moore-Penrose inverse in studying the numerical radius has become apparent.

 If $T\in\mathcal{B}(\mathcal{H})$, its Moore-Penrose inverse as one of generalized inverses is denoted by $T^{\dagger}.$ Generally, $T^{\dagger}$ is a densely defined unbounded operator. Yet, if $T$ has a closed range, then $T^{\dagger}\in\mathcal{B}(\mathcal{H})$. We will use the notation $\mathcal C\mathcal R (\mathcal{H})$ to denote the subset of $\mathcal{B}(\mathcal{H})$ of operators with closed ranges. If $T\in \mathcal C\mathcal R(\mathcal{H})$, the Moore-Penrose inverse $T^{\dagger}$ is the unique operator $G$ in $\mathcal{B}(\mathcal{H})$ that satisfies the following four equations, simultaneously:
\begin{equation*}\label{eq_properties_1}
TGT=T, GTG=G, {{\left( TG \right)}^{*}}=TG \;{\text{and}}\; {{\left( GT \right)}^{*}}=GT.
\end{equation*}
It is clear that when $\mathcal{H}$ is finite dimensional, all operators in $\mathcal{B}(\mathcal{H})$ have closed ranges. Thus, all matrices have the Moore-Penrose inverse. It is easy to see that, for $T\in \mathcal C\mathcal R(\mathcal{H}),$
\begin{equation}\label{7}
\begin{aligned}
T&=T{{T}^{\dagger }}T \\
& ={{\left( T{{T}^{\dagger }} \right)}^{*}}T \quad \text{(since $T{{T}^{\dagger }}$ is self-adjoint)}\\
& ={{\left( {{T}^{\dagger }} \right)}^{*}}{{\left| T \right|}^{2}}
\end{aligned}
\end{equation}
and
\begin{equation}\label{8}
\begin{aligned}
T&=T{{T}^{\dagger }}T \\
& =T{{\left( {{T}^{\dagger }}T \right)}^{*}} \quad \text{(since ${{T}^{\dagger }}T$ is self-adjoint)}\\
& ={{\left| T^{*} \right|}^{2}}{{\left( {{T}^{\dagger }} \right)}^{*}}.
\end{aligned}
\end{equation}

The reader is referred to \cite{Corach_NFAO_2005,Fongi_LAA_2023} for further reading on the Moore-Penrose inverse. Furthermore, the reader is referred to \cite{bhu1,Sababheh_CAOT_2024} for some discussion of numerical radii bounds via the Moore-Penrose inverse.

Although the following scalar inequalities are simple, we state them for completeness. For positive real numbers $ a, b, $ the classical Young's inequality says that if $ p\geq q > 1 $ are such that
$ \frac{1}{p} + \frac{1}{q}=1,$ then
$$ ab \leq \frac{a^{p}}{p} + \frac{b^{q}}{q}.$$
A refinement of this inequality was shown in \cite{kittanehmanasreh} as follows:
\begin{equation}\label{y1}
ab + \frac{1}{p}( a^{p/2}-b^{q/2})^{2} \leq \frac{a^{p}}{p} + \frac{b^{q}}{q},
\end{equation}
which has been generalized in \cite{man} by the form
\begin{equation}\label{y2}
( ab )^{r}+ \frac{1}{p}( a^{rp/2}-b^{rq/2})^{2} \leq \left( \frac{a^{p}}{p} + \frac{b^{q}}{q}\right)^{r}.
\end{equation}
where $ r=1, 2, \ldots. $

Next, we state some lemmas that we will need to accomplish our proofs. The first lemma is known as the McCarthy inequalities. In what follows, an operator $T\in\mathcal{B}(\mathcal{H})$ is said to be positive if $\left<Tx,x\right>\geq 0$ for all $x\in\mathcal{H}.$ For brevity, we may write $T\geq O$ to denote a positive operator $T$.
\begin{lemma}\cite{Kittaneh_PubRes_1988}\label{l1}
Let $T\in\mathcal{B}(\mathcal{H})$ be positive and $ x \in \mathcal{H} $ with $ \Vert x \Vert=1. $ Then\\

 $ (i)\, \langle Tx, x \rangle^{r} \leq \langle T^{r}x, x \rangle $ for $ r \geq 1;$

 $ (ii)\, \langle T^{r}x, x \rangle \leq \langle Tx, x \rangle^{r}$ for $ 0 < r \leq 1. $
\end{lemma}

 In the following, we have the converse of H\"{o}lder-McCarthy inequality as an extension of the Kantorovich inequality. For the used ordering, we recall that self-adjoint operators $S,T$ that satisfy $T-S\geq O$ are described by $T\geq S$.

\begin{lemma}\cite{pec}\label{l2}
Let $T\in\mathcal{B}(\mathcal{H})$ be a positive operator satisfying $ mI \leq T \leq MI $
for some scalars $ 0 < m< M. $ Then, for every unit vector $ x \in \mathcal{H} $,

 $ (i)\, \langle T^{r} x, x \rangle \leq K(m, M, r) \langle Tx, x \rangle^{r} $ for $r > 1;$

 $ (ii)\, K(m, M, r) \langle Tx, x \rangle^{r} \leq \langle T^{r}x, x \rangle$ for $ 0 \leq r \leq 1, $
where
$$ K(m, M, r)=\dfrac{(mM^{r}-Mm^{r})}{(r-1)(M-m)}\left( \dfrac{r-1}{r}\dfrac{M^{r}-m^{r}}{mM^{r}-Mm^{r}} \right)^{r} $$
for every $ r \in \mathbb{R}. $
\end{lemma}
We remark that $ K(m, M, r) $ is an extension of the Kantorovich constant $ \dfrac{(M+m)^{2}}{4mM}.$ In fact, $ K(m, M, 2) = K(m, M, -1)= \dfrac{(M+m)^{2}}{4mM}=\left( \dfrac{M\nabla m}{M \sharp m}\right)^{2}, $ where $M\nabla m=\frac{M+m}{2}$ and $M\sharp m=M^{1/2}m^{1/2}.$\\
The following lemma is known as the mixed Schwarz inequality.
\begin{lemma}\cite{Kato_MathAnn_1952}\label{l3}
Let $T\in\mathcal{B}(\mathcal{H})$ and $ x \in \mathcal{H}.$ Then
$$ \vert \langle Tx, x \rangle \vert^{2} \leq \langle \vert T \vert x, x \rangle \langle \vert T^{*} \vert x, x \rangle.$$
\end{lemma}

 \begin{lemma}\cite[Proposition 2.4]{bhu}\label{l4}
Let $ S, T, R \in \mathcal{B}(\mathcal{H})$ be such that $ S $ and $ T $ are positive. If the $ 2\times 2 $ operator matrix $\begin{bmatrix}
S & R^{*}\\
R & T
\end{bmatrix}$
is positive in $\mathcal{B}(\mathcal{H}\oplus\mathcal{H})$, then $ \omega^{2}(R )\leq \frac{1}{2} \Vert S \Vert \Vert T \Vert+ \frac{1}{2} \omega(ST). $
\end{lemma}

 \begin{lemma}\cite{Sababheh_CAOT_2024}\label{l5}
Let $ T \in \mathcal{CR}(\mathcal{H}). $ Then $ \vert \langle Tx, y \rangle \vert^{2} \leq \langle \vert T \vert^{2} x, x \rangle \langle T T^{\dagger} y, y \rangle $
for $ x, y \in \mathcal{H}.$
\end{lemma}

 \begin{lemma}\cite{Kittaneh_PubRes_1988}\label{l6}
Let $ A, B, C \in \mathcal{B}(\mathcal{H})$, be such that $ A$ and $ C $ are positive. Then
the $ 2\times 2 $ operator matrix $\begin{bmatrix}
A & B^{*}\\
B & C
\end{bmatrix}$
is positive if and only if $ \vert \langle Bx, y \rangle \vert^{2} \leq \langle A x, x \rangle \langle C y, y \rangle$ for all $x,y\in\mathcal{H}.$
\end{lemma}

 \begin{lemma}\cite[Theorem 1]{Kittaneh_PubRes_1988}\label{l-main}
Let $ T \in \mathcal{B}(\mathcal{H})$ and let $ x, y \in \mathcal{H}. $ If $ f $ and $ g $ are two nonnegative continuous functions on $ [0, \infty) $ satisfying $ f(t)g(t)=t$ for all $t \geq 0, $ then
$$ \vert \langle Tx, y \rangle\vert^{2} \leq \langle f^{2}(\vert T \vert)x, x \rangle \langle g^{2}(\vert T^{*} \vert)y, y \rangle.$$
\end{lemma}
In the sequel, we will repeatedly deal with continuous functions $f,g$, satisfying $f(t)g(t)=t$ for all $t\geq 0.$ For convenience, we will call such functions conjugate functions.

 Extending the McCarthy inequality, we have the following interesting lemma, in which $Sp(T)$ refers to the spectrum of $T$.
\begin{lemma}\cite[Theorem 1.2]{pec}\label{l-main-2}
Let $ T \in \mathcal{ B}(\mathcal{H})$ be a self-adjoint operator with $ Sp(T) \subseteq [m, M] $ for some scalars $ m < M. $ If $ h(t) $ is a convex function on $ [m, M], $ then
$$ h( \langle Tx, x\rangle ) \leq \langle h(T)x, x \rangle,$$
for every unit vector $ x \in \mathcal{H}.$
\end{lemma}

 Our target in this paper is to discuss further new bounds of the numerical radius, implementing the Moore-Penrose inverse. Our results will add a new set of bounds to the existing literature. While some bounds are new, and some refine selected known bounds.

\section{Main Results}\label{se2}
We begin our analysis in this section starting with some inner product results that involve the Moore-Penrose inverse, with direct applications towards numerical radii bounds.
\begin{lemma}\label{2-1}
Let $ T \in \mathcal{CR}(\mathcal{H})$ and $ p \geq q > 1,$ with $ \frac{1}{p}+ \frac{1}{q}=1.$
If $mI \leq TT^{*} \leq MI $ for some scalars $ 0<m<M, $ then for every unit vectors
$ x, y \in \mathcal{H},$ \\

 $$ \vert \langle Tx, y \rangle\vert + \delta(x, y) \leq \frac{1}{p}\left\langle (T^{\dagger} T)^{p/2}x, x \right\rangle +
\frac{1}{qK_{q/2}} \left\langle \vert T^{*} \vert^{q}y, y\right\rangle, $$
where $ \delta(x, y)=\frac{1}{p}\left( \left\langle T^{\dagger}Tx, x \right\rangle^{p/4}-\left\langle TT^{*}y, y \right\rangle^{q/4} \right)^{2}$ and $K_{q/2}=K(m, M, q/2).$\\
\end{lemma}
\begin{proof}
Let $ x, y \in \mathcal{H}.$ Using Cauchy- Schwarz inequality, inequality (\ref{y1}), Lemma \ref{l1}(i) and Lemma \ref{l2}(ii) respectively, we get
\begin{align*}
\left\vert \langle Tx, y \rangle \right\vert &= \left\vert \langle TT^{\dagger}Tx, y \rangle \right\vert \\
&= \left\vert \langle T^{\dagger}Tx, T^{*}y \rangle \right\vert\\
&\leq \Vert T^{\dagger}Tx \Vert \Vert T^{*}y \Vert\\
&\leq \frac{1}{p} \left\langle T^{\dagger}Tx, x \right\rangle^{p/2} + \frac{1}{q} \left\langle TT^{*}y, y \right\rangle^{q/2} -\delta(x, y)\\
&\leq \frac{1}{p} \left\langle (T^{\dagger}T)^{p/2}x, x \right\rangle + \frac{1}{qK_{q/2}} \left\langle \vert T^{*} \vert^{ q} y, y \right\rangle -\delta(x, y). \\
\end{align*}
This completes the proof.
\end{proof}



Choosing $ x=y $ and taking the supremum over all unit vectors $x \in\mathcal{H}$ in Lemma \ref{2-1}, we obtain the following bound for the numerical radius of an operator $T$.
\begin{theorem}\label{Thm_1_Norm_w}
Let $ T \in \mathcal{CR}(\mathcal{H})$ and $ p \geq q > 1,$ with $ \frac{1}{p}+ \frac{1}{q}=1.$
If $mI \leq TT^{*} \leq MI $ for some scalars $ 0<m<M, $ then
$$ \omega (T ) + \inf_{\Vert x \Vert=1} \delta(x) \leq \left\Vert \frac{1}{p}(T^{\dagger} T )^{p/2} +
\frac{1}{qK_{q/2}} \vert T^{*}\vert^{q} \right\Vert, $$
where $ \delta(x)=\frac{1}{p}\left( \left\langle T^{\dagger}Tx, x \right\rangle^{p/4}-\left\langle TT^{*}x, x \right\rangle^{q/4} \right)^{2} $ and $K_{q/2}=K(m, M, q/2). $

 \end{theorem}



We next identify the equality cases in Theorem\ref{Thm_1_Norm_w}, which is essential for sharp-bound applications.

\begin{remark}
Equality in Theorem \ref{Thm_1_Norm_w} holds if and only if there exists a unit vector 
$x_0 \in \mathcal{H}$ such that the following conditions are simultaneously 
satisfied:

(i) $x_0$ is a common eigenvector of $T^{\dagger}T$ and $TT^{*}$;

(ii) $x_0$ attains both the numerical radius of $T$ and the minimum of 
$\delta$, i.e., $|\langle T x_0, x_0 \rangle| = \omega(T)$ and 
$\delta(x_0) = \inf_{\|x\|=1} \delta(x)$;

(iii) the vectors $T^{\dagger}T x_0$ and $T^{*}x_0$ are linearly dependent;

(iv) the equality condition in Young's inequality holds, i.e., 
$\|T^{\dagger}T x_0\|^{p/2} = \|T^{*}x_0\|^{q/2}$; and

(v) $x_0$ is an eigenvector of 
$A := \frac{1}{p} (T^{\dagger}T)^{p/2} + \frac{1}{qK_{q/2}} |T^{*}|^{q}$ 
corresponding to its largest eigenvalue.

In the case $p = q = 2$, condition (iv) simplifies to 
$\|T^{\dagger}T x_0\| = \|T^{*}x_0\|$, and $K_1 = 1$. If $T$ is normal, then 
$T^{\dagger}T$ and $TT^{*}$ commute; hence they possess a common system of 
eigenvectors. Consequently, condition (i) is automatically satisfied whenever 
$x_0$ is an eigenvector of either operator. If $T$ is invertible, then 
$T^{\dagger}T = I$, and condition (i) holds trivially. 
Thus, the bound in Theorem \ref{Thm_1_Norm_w} is generally strict, and equality occurs only 
under very restrictive assumptions.
\end{remark}

\begin{remark}\label{r1}
It is remarkable that by \cite[Theorem 2.54]{pec} for $ p=q=2, $ we have $ K_{1}=K(m,M,1)=1 $ and
since from $ p \geq q > 1 $ with $ \frac{1}{p}+ \frac{1}{q}=1,$ we get $ \frac{1}{2} < \frac{q}{2} \leq 1.$ Therefore $\frac{1}{K_{q/2}} \geq 1. $
\end{remark}
\begin{remark}\label{r2}
For $ T \in \mathcal{CR}(\mathcal{H}),$ it is not difficult to see that the condition $O <mI \leq TT^{\dagger} \leq MI $ is equivalent to $TT^{\dagger} =I $ and $ 0<m \leq 1 \leq M,$ where $I$ is the identity operator in $\mathcal{B}(\mathcal{H}).$
This follows because $TT^{\dagger} $ is a positive projection, so by positivity, $ \sigma (TT^{\dagger}) ={1}.$ Then $TT^{\dagger} =I. $ Also specially if $ T \in \mathbb{M}_{n}, $ then $TT^{\dagger} =I$ implies $T^{\dagger}=T^{-1}. $
\end{remark}
\begin{remark}\label{r1-3}
For $ T \in \mathcal{CR}(\mathcal{H}),$ it is easy to see that if
$ \overline{W(T^{\dagger}T)}^{p/4} \bigcap \overline{W(TT^{*})}^{q/4}= \emptyset,$ then $\inf_{\Vert x \Vert=1} \delta(x) > 0.$
In particular, it is obvious that for $ p=q=2, $
$$ \inf_{\Vert x \Vert=1} \delta(x) > 0 \Leftrightarrow 0 \notin \overline{W(T^{\dagger}T-TT^{*})}.$$
\end{remark}
\color{black}
\begin{example}
Set
$ T= \begin{bmatrix}
0 &2\\
3 & 0
\end{bmatrix},$ and $ p=3, q=\frac{3}{2}.$ Then $ \overline{W(T^{\dagger}T)}^{p/4} \bigcap \overline{W(TT^{*})}^{q/4}= \lbrace 1 \rbrace \bigcap [2^{3/4}, 3^{3/4}]= \emptyset.$ So, in this case $ \inf_{\Vert x \Vert=1} \delta(x) > 0. $
\end{example}


 In line with the proof of Lemma \ref{2-1}, using Cauchy- Schwarz inequality, inequality (\ref{y2}), Lemma \ref{l1}(i) and convexity of $ f(t)=t^{r} $ for $ r=2, 3, \ldots $ on $ [0, \infty), $ respectively, we obtain the following lemma.
\begin{lemma}\label{3-1}
Let $ T \in \mathcal{CR}(\mathcal{H})$ and $ p \geq q > 1,$ with $ \frac{1}{p}+ \frac{1}{q}=1.$
Then for every unit vectors $ x, y \in \mathcal{H},$ and $ r=2, 3, \ldots, $\\

 $$ \vert \langle Tx, y \rangle\vert^{r} + \delta(x, y) \leq \frac{1}{p}\left\langle (T^{\dagger} T)^{rp/2}x, x \right\rangle +
\frac{1}{q} \left\langle \vert T^{*} \vert^{rq}y, y\right\rangle, $$
where $ \delta(x, y)=\frac{1}{p}\left( \left\langle T^{\dagger}Tx, x \right\rangle^{rp/4}-\left\langle TT^{*}y, y \right\rangle^{rq/4} \right)^{2}.$
\end{lemma}
In view of Lemma \ref{3-1}, the following interesting inequality is obtained.
\begin{proposition}
Let $ T \in \mathcal{CR}(\mathcal{H})$ and $ p \geq q > 1,$ with $ \frac{1}{p}+ \frac{1}{q}=1.$
If $ r=2, 3, \ldots, $ then\\
$$ \omega^{r} (T ) + \inf_{\Vert x \Vert=1} \delta(x) \leq \left\Vert \frac{1}{p}(T^{\dagger} T )^{rp/2} +
\frac{1}{q} \vert T^{*}\vert^{rq} \right\Vert, $$
where $ \delta(x)=\frac{1}{p}\left( \left\langle T^{\dagger}Tx, x \right\rangle^{rp/4}-\left\langle TT^{*}x, x \right\rangle^{rq/4} \right)^{2}. $
\end{proposition}


The following theorem is a refined and generalized form of the inequalities of \cite[Theorem 2.1]{bhu1}.
\begin{theorem}\label{t1}
Let $ T, S \in \mathcal{CR}(\mathcal{H})$ and $ x, y \in \mathcal{H}$ be arbitrary vectors. If $ f $ and $ g $ are conjugate functions, then
\begin{equation}\label{eq.1}
\vert \langle (T+S)x, y \rangle\vert \leq \sqrt{\left\langle \left( f^{2}(\vert S \vert)+ T^{\dagger}Tg^{2}(\vert T \vert )T^{\dagger}T\right)x, x \right\rangle
\left\langle \left( SS^{\dagger}g^{2}(\vert S^{*} \vert ) SS^{\dagger} +f^{2}(\vert T^{*}\vert) \right)y, y\right\rangle},
\end{equation}
and
\begin{equation}\label{eq.2}
\vert \langle (T+S)x, y \rangle\vert \leq \sqrt{\left\langle \left( f^{2}(\vert T \vert)+ S^{\dagger}Sg^{2}(\vert S \vert )S^{\dagger}S\right)x, x \right\rangle
\left\langle \left( TT^{\dagger}g^{2}(\vert T^{*} \vert ) TT^{\dagger} +f^{2}(\vert S^{*}\vert) \right)y, y\right\rangle}.
\end{equation}

 \end{theorem}

 \begin{proof}
Let $ x, y \in \mathcal{H}.$ Using Cauchy- Schwarz inequality, and Lemma \ref{l-main} respectively, we get
\begin{align*}
2 \left\vert \langle Sx, y \rangle \right\vert &= 2 \left\vert \langle SS^{\dagger}Sx, y \rangle \right\vert \\
&= 2 \left\vert \langle Sx, SS^{\dagger}y \rangle \right\vert\\
&\leq 2 \left \langle f^{2}(\vert S \vert) x, x \right \rangle^{1/2} . \left \langle SS^{\dagger}g^{2}(\vert S^{*} \vert ) SS^{\dagger}y, y \right \rangle^{1/2} \\
&\leq \left \langle f^{2}(\vert S \vert) x, x \right \rangle + \left \langle SS^{\dagger}g^{2}(\vert S^{*} \vert ) SS^{\dagger}y, y \right \rangle. \\
\end{align*}
This yields
\begin{align*}
\left\langle
\begin{bmatrix}
f^{2}(\vert S \vert) & S^{*}\\
S & SS^{\dagger}g^{2}(\vert S^{*} \vert ) SS^{\dagger}
\end{bmatrix}
\begin{bmatrix}
x \\
y
\end{bmatrix},
\begin{bmatrix}
x \\
y
\end{bmatrix}
\right\rangle &= \langle f^{2}(\vert S \vert) x, x \rangle+ \langle S^{
*}y, x \rangle\\
&\quad + \langle Sx, y \rangle+\langle SS^{\dagger}g^{2}(\vert S^{*} \vert ) SS^{\dagger}y, y \rangle\\
& \geq 2\vert \langle Sx, y \rangle \vert +2\Re \langle Sx, y \rangle\\
& \geq 0.
\end{align*}

 Thus, $ \begin{bmatrix}
f^{2}(\vert S \vert) & S^{*}\\
S & SS^{\dagger}g^{2}(\vert S^{*} \vert ) SS^{\dagger}
\end{bmatrix} \geq O. $ Similarly, we can also show that \\
$ \begin{bmatrix}
T^{\dagger}Tg^{2}(\vert T \vert )T^{\dagger}T & T^{*}\\
T & f^{2}(\vert T^{*}\vert)
\end{bmatrix} \geq O. $ Therefore,
$$ \begin{bmatrix}
f^{2}(\vert S \vert)+ T^{\dagger}Tg^{2}(\vert T \vert )T^{\dagger}T & (S+T)^{*}\\
S+T & SS^{\dagger}g^{2}(\vert S^{*} \vert ) SS^{\dagger} +f^{2}(\vert T^{*}\vert)
\end{bmatrix} \geq O. $$
Now, by Lemma \ref{l6}, we have (\ref{eq.1}). Interchanging $ S $ and $ T, $
we obtain (\ref{eq.2}), as desired.
\end{proof}

 By setting $ f(t)=t $ and $ g(t)=1 $ in Theorem \ref{t1}, we obtain the following result.
\begin{corollary}\cite[Theorem 2.1]{bhu1}
If $ T, S \in \mathcal{CR}(\mathcal{H})$ and $ x, y \in \mathcal{H}$ are arbitrary vectors, then
\begin{equation*}
\left| \langle (T+S)x, y \rangle \right|
\le
\sqrt{
\left\langle \left( \lvert S \rvert^{2} + T^{\dagger} T \right)x, x \right\rangle
\left\langle \left( S S^{\dagger} + \lvert T^{*} \rvert^{2} \right)y, y \right\rangle
},
\end{equation*}
and
\begin{equation*}
\vert \langle (T+S)x, y \rangle\vert \leq \sqrt{\left\langle \left( \vert T \vert^{2}+ S^{\dagger}S\right)x, x \right\rangle
\left\langle \left( TT^{\dagger} +\vert S^{*}\vert^{2} \right)y, y\right\rangle}.
\end{equation*}
\end{corollary}

Choosing $ f(t)=g(t)=t^{1/2} $ in (\ref{eq.1}) and (\ref{eq.2}), respectively, we have the following result.

 \begin{corollary}\label{Cor_Ned_01}
If $ T, S \in \mathcal{CR}(\mathcal{H})$ and $ x, y \in \mathcal{H}$ are arbitrary vectors, then
\begin{equation*}
\vert \langle (T+S)x, y \rangle\vert \leq \sqrt{\left\langle \left( \vert S \vert+ T^{\dagger}T\vert T \vert T^{\dagger}T\right)x, x \right\rangle
\left\langle \left( SS^{\dagger}\vert S^{*} \vert SS^{\dagger} +\vert T^{*}\vert \right)y, y\right\rangle},
\end{equation*}
and
\begin{equation*}
\vert \langle (T+S)x, y \rangle\vert \leq \sqrt{\left\langle \left( \vert T \vert + S^{\dagger}S\vert S \vert S^{\dagger}S\right)x, x \right\rangle
\left\langle \left( TT^{\dagger}\vert T^{*} \vert TT^{\dagger} +\vert S^{*}\vert \right)y, y\right\rangle}.
\end{equation*}
\end{corollary}
Upon taking the supremum over unit vectors $x,y$ in Corollary \ref{Cor_Ned_01}, we reach the following norm version.
\begin{corollary}
Let $ T, S \in \mathcal{CR}(\mathcal{H}),$ and let $ f $ and $ g $ be conjugate functions. Then
\begin{equation*}
\lVert T + S \rVert
\le
\sqrt{
\left\lVert
f^{2}(\lvert S \rvert)
+ T^{\dagger} T\, g^{2}(\lvert T \rvert)\, T^{\dagger} T
\right\rVert
\,
\left\lVert
S S^{\dagger}\, g^{2}(\lvert S^{*} \rvert)\, S S^{\dagger}
+ f^{2}(\lvert T^{*} \rvert)
\right\rVert
},
\end{equation*}
and
\begin{equation*}
\lVert T + S \rVert
\le
\sqrt{
\left\lVert
f^{2}(\lvert T \rvert)
+ S^{\dagger} S\, g^{2}(\lvert S \rvert)\, S^{\dagger} S
\right\rVert
\,
\left\lVert
T T^{\dagger}\, g^{2}(\lvert T^{*} \rvert)\, T T^{\dagger}
+ f^{2}(\lvert S^{*} \rvert)
\right\rVert
}.
\end{equation*}
\end{corollary}
\begin{remark}
Let $ A, B \in \mathcal{B}(\mathcal{H})$ be positive operators and $ 0 \leq \nu \leq 1.$ The weighted arithmetic mean, weighted harmonic mean and weighted geometric mean are defined respectively, as
$ A\nabla_{\nu} B=(1-\nu )A+\nu B, \, A!_{\nu} B=((1-\nu )A^{-1}+\nu B^{-1})^{-1} $ and
$ A \sharp_{\nu} B= A^{1/2}(A^{-1/2}BA^{-1/2})^{\nu}A^{1/2}.$ When $\nu=\frac{1}{2},$ it is customary to write $\nabla,!$ and $\sharp$ instead of $\nabla_{\frac{1}{2}}, !_{\frac{1}{2}}$ and $\sharp_{\frac{1}{2}}$, respectively.

 Let $ \sigma $ be an operator mean with representing function $ f(t). $ The operator mean with representing function $ t/f(t) $ is called the dual of $ \sigma $ and denoted by $ \sigma^{\perp}. $ It follows that
$$ \sigma^{\perp}=(\sigma^{0})^{*}= (\sigma^{*})^{0} \, \quad \text{and} \, \quad \sigma^{0}=(\sigma^{*})^{\perp}= (\sigma^{\perp})^{*}. $$
For example, $ \nabla_{\nu}^{\bot}=!_{1-\nu}, \sharp_{\nu}^{\bot} =\sharp_{1-\nu}.$
Now, if we replace $ f, g $ in Theorem \ref{t1}, respectively, with the representing function $ \sigma $ and $ \sigma^{\perp}, $ then for every operator mean $ \sigma, $ we have
\begin{small}
\begin{equation*}
\vert \langle (T+S)x, y \rangle\vert \leq \sqrt{\left\langle \left( ( I \sigma \vert S \vert)^{2}+ T^{\dagger}T(I \sigma^{\perp} \vert T \vert )^{2}T^{\dagger}T\right)x, x \right\rangle
\left\langle \left( SS^{\dagger}(I \sigma^{\perp}\vert S^{*} \vert )^{2} SS^{\dagger} +(I \sigma \vert T^{*}\vert)^{2} \right)y, y\right\rangle},
\end{equation*}
\end{small}
where $ T, S \in \mathcal{CR}(\mathcal{H})$ and $ x, y \in \mathcal{H}$ are arbitrary vectors in $ \mathcal{H}. $ In particular,
\begin{equation*}
\lVert T + S \rVert
\le
\sqrt{
\left\lVert
(I \sigma \lvert S \rvert)^{2}
+ T^{\dagger} T\, ( I \sigma^{\perp}\lvert T \rvert)^{2}\, T^{\dagger} T
\right\rVert
\,
\left\lVert
S S^{\dagger}\, (I \sigma^{\perp}\lvert S^{*} \rvert)^{2}\, S S^{\dagger}
+ (I \sigma \lvert T^{*} \rvert)^{2}
\right\rVert
}.
\end{equation*}
\end{remark}
On the other hand, we have the following refined and generalized form of \cite[Corollary 2.2]{bhu1}.
\begin{proposition}\label{p-2}
Let $ T, S \in \mathcal{CR}(\mathcal{H}),$ and let $ f $ and $ g $ be conjugate functions. Then
\begin{equation}\label{eq.5}
\omega (T+S) \leq \frac{1}{ 2}\left\Vert f^{2}(\vert S \vert)+ T^{\dagger}Tg^{2}(\vert T \vert )T^{\dagger}T+ SS^{\dagger}g^{2}(\vert S^{*} \vert ) SS^{\dagger} +f^{2}(\vert T^{*}\vert) \right\Vert.
\end{equation}
\end{proposition}
\begin{proof}
From Theorem \ref{t1}, for every unit vector $ x \in \mathcal{H}, $ we have
\begin{align*}
&\vert \langle (T+S)x, x \rangle\vert \\
&\leq \sqrt{\left\langle \left( f^{2}(\vert S \vert)+ T^{\dagger}Tg^{2}(\vert T \vert )T^{\dagger}T\right)x, x \right\rangle
\left\langle \left( SS^{\dagger}g^{2}(\vert S^{*} \vert ) SS^{\dagger} +f^{2}(\vert T^{*}\vert) \right)x, x\right\rangle}\\
& \leq \frac{1}{2} \left( \left\langle \left( f^{2}(\vert S \vert)+ T^{\dagger}Tg^{2}(\vert T \vert )T^{\dagger}T\right)x, x \right\rangle +\left\langle \left( SS^{\dagger}g^{2}(\vert S^{*} \vert ) SS^{\dagger} +f^{2}(\vert T^{*}\vert) \right)x, x\right\rangle
\right)\\
& \leq \frac{1}{ 2}\left\Vert f^{2}(\vert S \vert)+ T^{\dagger}Tg^{2}(\vert T \vert )T^{\dagger}T+ SS^{\dagger}g^{2}(\vert S^{*} \vert ) SS^{\dagger} +f^{2}(\vert T^{*}\vert) \right\Vert.
\end{align*}
Therefore, taking the supremum over $ \Vert x \Vert=1$ implies the desired result.
\end{proof}

 In the following example we show that Proposition \ref{p-2} can provide sharper bounds than \cite[Corollary 2.2]{bhu1}.
\begin{example}
Let $ T=\begin{bmatrix}
1&0\\
0&1/2
\end{bmatrix}, S=\begin{bmatrix}
0&2\\
1&0
\end{bmatrix}, f(t)=\frac{1+t}{2} ($ the representing function of $ \nabla ) $ and $ g(t)=\frac{2t}{1+t} ($the representing function of $ ! ).$ Then it can be shown with a simple calculation that $ \omega(T+S)\approx 2.27 $ and the right hand side of Proposition \ref{p-2}
is almost $ 2.97,$ but the right hand side of \cite[Corollary 2.2]{bhu1} is $ 3.5, $ showing how Proposition \ref{p-2} can provide better estimates than \cite[Corollary 2.2]{bhu1}.

 \end{example}
The following numerical radius bound follows from Proposition \ref{p-2}, by choosing $ S=T.$
\begin{corollary}
Let $ T \in \mathcal{CR}(\mathcal{H})$ and $ f $ and $ g $ be conjugate functions. Then

 \begin{equation}\label{eq.7}
\omega (T) \leq \frac{1}{4}\left\Vert f^{2}(\vert T \vert)+ T^{\dagger}Tg^{2}(\vert T \vert )T^{\dagger}T+ TT^{\dagger}g^{2}(\vert T^{*} \vert ) TT^{\dagger} +f^{2}(\vert T^{*}\vert) \right\Vert.
\end{equation}
In particular, if $ T $ is invertible, then

 \begin{equation}\label{eq.8}
\omega (T) \leq \frac{1}{4}\left\Vert f^{2}(\vert T \vert)+ g^{2}(\vert T \vert )+ g^{2}(\vert T^{*} \vert ) +f^{2}(\vert T^{*}\vert) \right\Vert.
\end{equation}

 \end{corollary}
Considering $ S=0$ and $ T=0,$ respectively, in Proposition \ref{p-2}, we obtain the following numerical radius inequality that is a refined and generalized version of \cite[Corollary 2.3]{bhu1}.
\begin{corollary}
Let $ T \in \mathcal{CR}(\mathcal{H})$ and $ f $ and $ g $ be conjugate functions. Then

 \begin{equation}
\omega (T) \leq \frac{1}{2}
\min \left\{
\left\Vert T^{\dagger}T g^{2}(\lvert T \rvert ) T^{\dagger}T
+ f^{2}(\lvert T^{*}\rvert) \right\Vert,\,
\left\Vert f^{2}(\lvert T \rvert)
+ TT^{\dagger} g^{2}(\lvert T^{*} \rvert ) TT^{\dagger} \right\Vert
\right\}.
\end{equation}
In particular, if $ T $ is invertible, then

 \begin{equation}
\omega (T) \leq \frac{1}{2} \min \left\lbrace \left\Vert g^{2}(\vert T \vert )+f^{2}(\vert T^{*}\vert) \right\Vert, \left\Vert f^{2}(\vert T \vert)+ g^{2}(\vert T^{*} \vert ) \right\Vert \right\rbrace .
\end{equation}

 \end{corollary}


\color{black}
The following theorem is a refined and generalized form of the inequalities of Theorem 2.2 in \cite{bhu1}.
\begin{theorem}\label{t2}
Let $ T, S \in \mathcal{CR}(\mathcal{H})$ and $ f $ and $ g $ be conjugate functions. Then
\begin{align*}
\omega^{2} (T+S) &\leq \frac{1}{2} \left\Vert f^{2}(\vert S \vert)+ T^{\dagger}Tg^{2}(\vert T \vert )T^{\dagger}T \right\Vert \left\Vert SS^{\dagger}g^{2}(\vert S^{*} \vert ) SS^{\dagger} +f^{2}(\vert T^{*}\vert) \right\Vert\\
&+\frac{1}{2} \omega \left( \left( f^{2}(\vert S \vert)+ T^{\dagger}Tg^{2}(\vert T \vert )T^{\dagger}T \right) \left( SS^{\dagger}g^{2}(\vert S^{*} \vert ) SS^{\dagger} +f^{2}(\vert T^{*}\vert) \right) \right).
\end{align*}

 \end{theorem}
\begin{proof}
From of the proof of Theorem \ref{t1}, we have
$$ \begin{bmatrix}
f^{2}(\vert S \vert)+ T^{\dagger}Tg^{2}(\vert T \vert )T^{\dagger}T & (S+T)^{*}\\
S+T & SS^{\dagger}g^{2}(\vert S^{*} \vert ) SS^{\dagger} +f^{2}(\vert T^{*}\vert)
\end{bmatrix} \geq O. $$
Therefore, from Lemma \ref{l4}, the proof is complete.
\end{proof}

 By choosing $ T=S$ in Theorem \ref{t2}, the following numerical radius bound follows.
\begin{corollary}
Let $ T \in \mathcal{CR}(\mathcal{H}),$ and let $ f $ and $ g $ be conjugate functions. Then
\begin{small}
\begin{align*}
\omega^{2} (T) &\leq \frac{1}{8} \left\Vert f^{2}(\vert T \vert)+ T^{\dagger}Tg^{2}(\vert T \vert )T^{\dagger}T \right\Vert \left\Vert TT^{\dagger}g^{2}(\vert T^{*} \vert ) TT^{\dagger} +f^{2}(\vert T^{*}\vert) \right\Vert\\
&\quad\quad+\frac{1}{8} \omega \left( \left( f^{2}(\vert T \vert)+ T^{\dagger}Tg^{2}(\vert T \vert )T^{\dagger}T \right) \left( TT^{\dagger}g^{2}(\vert T^{*} \vert ) TT^{\dagger} +f^{2}(\vert T^{*}\vert) \right) \right).
\end{align*}
\end{small}

 \end{corollary}

\begin{corollary}
Let $ T \in \mathcal{B}(\mathcal{H})$ be invertible, and let $ f $ and $ g $ be conjugate functions. Then
\begin{small}
\begin{align*}
\omega^{2} (T) &\leq \frac{1}{8} \left\Vert f^{2}(\vert T \vert)+ g^{2}(\vert T \vert ) \right\Vert \left\Vert g^{2}(\vert T^{*} \vert ) +f^{2}(\vert T^{*}\vert) \right\Vert\\
&+\frac{1}{8} \omega \left( \left( f^{2}(\vert T \vert)+ g^{2}(\vert T \vert ) \right) \left( g^{2}(\vert T^{*} \vert ) +f^{2}(\vert T^{*}\vert) \right) \right).
\end{align*}
\end{small}
\end{corollary}

 \begin{theorem}
Let $ T \in \mathcal{CR}(\mathcal{H})$ and $ f $ and $ g $ be conjugate functions. If $ h $ is an increasing convex function on $ [0, \infty), $ then

 \begin{equation*}
h(\omega(T)) \leq \frac{1}{2} \Vert h(f^{2}(\vert T \vert)) + h (TT^{\dagger}g^{2}( \vert T^{*} \vert ) TT^{\dagger} ) \Vert.
\end{equation*}
In particular, if $ h(t)=t^{r}$ with $r\geq 1, $ then
\begin{equation*}
\omega^{r}(T) \leq \frac{1}{2} \Vert f^{2r}(\vert T \vert) + (TT^{\dagger}g^{2}( \vert T^{*} \vert ) TT^{\dagger} )^{r} \Vert.
\end{equation*}
\end{theorem}
\begin{proof}
Using Lemma \ref{l-main}, monotonicity of $ h, $ the arithmetic-geometric mean inequality, convexity of $ h $ and Lemma \ref{l-main-2}, respectively, we have
\begin{align*}
h(\vert \langle Tx, x \rangle \vert) &= h(\vert \langle Tx, TT^{\dagger}x \rangle \vert)\\
& \leq h( \langle f^{2}(\vert T \vert ) x, x \rangle^{1/2} \langle TT^{\dagger}g^{2}(\vert T^{*} \vert ) TT^{\dagger}x, x \rangle^{1/2} )\\
& \leq h \left( \dfrac{\langle f^{2}(\vert T \vert ) x, x \rangle + \langle TT^{\dagger}g^{2}(\vert T^{*} \vert ) TT^{\dagger}x, x \rangle}{2} \right)\\
& \leq \dfrac{ h\left( \langle f^{2}(\vert T \vert ) x, x \rangle \right) + h\left(\langle TT^{\dagger}g^{2}(\vert T^{*} \vert ) TT^{\dagger}x, x \rangle \right) }{2} \\
& \leq \dfrac{ \langle h\left(f^{2}(\vert T \vert )\right) x, x \rangle + \langle h\left(TT^{\dagger}g^{2}(\vert T^{*} \vert ) TT^{\dagger}\right) x, x \rangle }{2} \\
&= \dfrac{1}{2} \langle \left[ h\left(f^{2}(\vert T \vert )\right)+ h\left(TT^{\dagger}g^{2}(\vert T^{*} \vert ) TT^{\dagger}\right) \right] x, x \rangle.
\end{align*}
Therefore, taking the supremum over $ \Vert x \Vert=1$ implies the desired result.
\end{proof}

\section{Some generalized inequalities for the Hadamard product}\label{3}
For the separable Hilbert space $ (\mathcal{H}, \langle \cdot , \cdot \rangle) $, another important product is the Hadamard product, denoted by $ T\circ S $, and
defined by $$ \langle (T \circ S )e_{i}, e_{j}\rangle= \langle Te_{i}, e_{j} \rangle \langle S e_{i}, e_{j} \rangle, $$
where $ \lbrace e_{j} \rbrace$ is an orthonormal basis of $ \mathcal{H}.$
We say $ T \in \mathcal{B}(\mathcal{H})$ is contraction if $ \Vert T \Vert \leq 1. $\\
\begin{lemma}\cite[Proposition 1.3.2]{bhat}\label{l-01}
Let $ T, S \in \mathcal{B}(\mathcal{H}) $ be positive operators. Then the operator matrix
$ \begin{bmatrix}
T & X\\
X^{*}& S
\end{bmatrix} $ is positive if and only if $ X=T^{1/2}C S^{1/2} $ for some contraction $ C. $
\end{lemma}

 \begin{lemma}\cite[Theorem 1.3.3]{bhat}\label{l-02}
Let $ T, S \in \mathcal{B}(\mathcal{H}) $ be strictly positive. Then the operator matrix
$ \begin{bmatrix}
T & X\\
X^{*}& S
\end{bmatrix} $ is positive if and only if $ T \geq X S^{-1} X^{*}.$
\end{lemma}
\begin{lemma}\cite[Theorem 4.1.3]{bhat}\label{l-03}
Let $ T, S \in \mathcal{B}(\mathcal{H}) $ be two positive operators. Then
$$ T \sharp S= \max \lbrace X: \, X=X^{*}, \begin{bmatrix}
T & X\\
X^{*}& S
\end{bmatrix} \geq O \rbrace.$$
\end{lemma}

 \color{black}
\begin{proposition}
Let $ T, S \in \mathcal{CR}(\mathcal{H}),$ and let $ f $ and $ g $ be conjugate functions. Then

 $$(S+T)= \sqrt{SS^{\dagger}g^{2}(\vert S^{*} \vert ) SS^{\dagger} +f^{2}(\vert T^{*}\vert)}C\sqrt{f^{2}(\vert S \vert)+ T^{\dagger}Tg^{2}(\vert T \vert )T^{\dagger}T},$$
for some contraction $ C$. Moreover,
\[
\Bigl( f^{2}(\vert S \vert)+ T^{\dagger}Tg^{2}(\vert T \vert )T^{\dagger}T\Bigr)
\geq
(S+T)^{*}\Bigl(SS^{\dagger}g^{2}(\vert S^{*} \vert ) SS^{\dagger} +f^{2}(\vert T^{*}\vert) \Bigr)^{-1}(S+T).
\]
In addition, if $ T, S $ are self-adjoint, then
\[
( S + T )\leq
\Bigl( f^{2}(\vert S \vert)+ T^{\dagger}Tg^{2}(\vert T \vert )T^{\dagger}T\Bigr)
\sharp
\Bigl(SS^{\dagger}g^{2}(\vert S^{*} \vert ) SS^{\dagger} +f^{2}(\vert T^{*}\vert) \Bigr).
\]
\end{proposition}

 \begin{proof}
From the positive $ 2\times 2 $ operator matrix in the proof of Theorem \ref{t2}, and by Lemma \ref{l-01}, Lemma \ref{l-02}, and Lemma \ref{l-03}, respectively, we obtain the desired result.
\end{proof}

\begin{proposition}\label{p-3-1}
Let $ T, S \in \mathcal{CR}(\mathcal{H}),$ and let $ f $ and $ g $ be conjugate functions. Then for every unit vectors
$ x, y \in \mathcal{H},$
\begin{equation*}
\vert \langle (T\circ S) x, y \rangle\vert^{2} \leq \left\langle \left( f^{2}(\vert S \vert) \circ (T^{\dagger}Tg^{2}(\vert T \vert )T^{\dagger}T)\right)x, x \right\rangle
\left\langle \left((SS^{\dagger}g^{2}(\vert S^{*} \vert ) SS^{\dagger}) \circ f^{2}(\vert T^{*}\vert) \right)y, y\right\rangle.
\end{equation*}

 In particular,
$$ \left\Vert T\circ S \right\Vert^{2} \leq \left\Vert f^{2}(\vert S \vert) \circ \left( T^{\dagger}Tg^{2}(\vert T \vert )T^{\dagger}T\right) \right\Vert . \left\Vert \left(SS^{\dagger}g^{2}(\vert S^{*} \vert ) SS^{\dagger}\right) \circ f^{2}(\vert T^{*}\vert) \right \Vert. $$
\end{proposition}
\begin{proof}
From the proof of Theorem \ref{t1}, we conclude that
$$ \begin{bmatrix}
f^{2}(\vert S \vert) \circ T^{\dagger}Tg^{2}(\vert T \vert )T^{\dagger}T & (S \circ T)^{*}\\
S \circ T & SS^{\dagger}g^{2}(\vert S^{*} \vert ) SS^{\dagger} \circ f^{2}(\vert T^{*}\vert)
\end{bmatrix} \geq O. $$
Now, by Lemma \ref{l6}, the result holds.
\end{proof}


 \begin{proposition}
Let $ T, S \in \mathcal{CR}(\mathcal{H}),$ and let $ f $ and $ g $ be conjugate functions. Then
\begin{equation}\label{eq.5}
\omega (T\circ S) \leq \frac{1}{2} \left\Vert \left( f^{2}(\vert S \vert) \circ (T^{\dagger}Tg^{2}(\vert T \vert )T^{\dagger}T)\right)+ \left((SS^{\dagger}g^{2}(\vert S^{*} \vert ) SS^{\dagger}) \circ f^{2}(\vert T^{*}\vert) \right) \right\Vert.
\end{equation}

 \end{proposition}
\begin{proof}
From Proposition \ref{p-3-1}, for every unit vector $ x \in \mathcal{H},$ we have
\begin{align*}
& \vert \langle (T\circ S)x, x \rangle\vert \leq \sqrt{\left\langle \left( f^{2}(\vert S \vert) \circ (T^{\dagger}Tg^{2}(\vert T \vert )T^{\dagger}T)\right)x, x \right\rangle
\left\langle (SS^{\dagger}g^{2}(\vert S^{*} \vert ) SS^{\dagger}) \circ f^{2}(\vert T^{*}\vert) x, x\right\rangle}\\
& \leq \frac{1}{2} \left( \left\langle \left( f^{2}(\vert S \vert) \circ (T^{\dagger}Tg^{2}(\vert T \vert )T^{\dagger}T)\right)x, x \right\rangle +
\left\langle (SS^{\dagger}g^{2}(\vert S^{*} \vert ) SS^{\dagger}) \circ f^{2}(\vert T^{*}\vert) x, x\right\rangle \right)\\
& \leq \frac{1}{2} \left\Vert \left( f^{2}(\vert S \vert) \circ (T^{\dagger}Tg^{2}(\vert T \vert )T^{\dagger}T)\right) + \left( (SS^{\dagger}g^{2}(\vert S^{*} \vert ) SS^{\dagger}) \circ f^{2}(\vert T^{*}\vert)\right) \right\Vert.
\end{align*}
Therefore, taking the supremum over all unit vectors $x$, we get (\ref{eq.5}).
\end{proof}

 Again from the proof of Theorem \ref{t1}, we have
$$ \begin{bmatrix}
f^{2}(\vert S \vert) \circ T^{\dagger}Tg^{2}(\vert T \vert )T^{\dagger}T & (S \circ T)^{*}\\
S \circ T & SS^{\dagger}g^{2}(\vert S^{*} \vert ) SS^{\dagger} \circ f^{2}(\vert T^{*}\vert)
\end{bmatrix} \geq O. $$
Now, by Lemma \ref{l-01}, Lemma \ref{l-02} and Lemma \ref{l-03}, respectively, we have the following corollary.
\begin{corollary}
Let $ T, S \in \mathcal{CR}(\mathcal{H}),$ and let $ f $ and $ g $ be conjugate functions. Then
$$(S \circ T)= \sqrt{(SS^{\dagger}g^{2}(\vert S^{*} \vert ) SS^{\dagger}) \circ f^{2}(\vert T^{*}\vert)}C\sqrt{f^{2}(\vert S \vert) \circ (T^{\dagger}Tg^{2}(\vert T \vert )T^{\dagger}T)},$$
for some contraction $ C$. Moreover,
\[
\Bigl( f^{2}(\vert S \vert) \circ (T^{\dagger}Tg^{2}(\vert T \vert )T^{\dagger}T) \Bigr)
\geq
(S \circ T)^{*}\Bigl((SS^{\dagger}g^{2}(\vert S^{*} \vert ) SS^{\dagger}) \circ f^{2}(\vert T^{*}\vert) \Bigr)^{-1}(S+T).
\]
In addition, if $ T, S $ are self-adjoint, then
\[
( S \circ T )\leq
\Bigl( f^{2}(\vert S \vert) \circ (T^{\dagger}Tg^{2}(\vert T \vert )T^{\dagger}T) \Bigr)
\sharp
\Bigl( (SS^{\dagger}g^{2}(\vert S^{*} \vert ) SS^{\dagger}) \circ f^{2}(\vert T^{*}\vert) \Bigr).
\]
\end{corollary}
\section{Additional results}\label{se4}
In this short section we present some more elaborated results that involve norm and numerical radii bounds.
\begin{lemma}\label{l7}
Let $ T \in \mathcal{CR}(\mathcal{H})$ and $ p\geq q >1 $ be conjugate exponents. Then for all
$ x \in \mathcal{H}$ with $ \Vert x \Vert=1, $
$$ \Vert Tx \Vert \Vert T^{*}x\Vert \leq \left\Vert \frac{1}{p} (T^{*}T)^{p}+ \frac{1}{q} (TT^{*})^{q}\right \Vert^{1/2}-\inf_{\Vert x \Vert=1}\mu(x), $$
where $ \mu(x)= \frac{1}{p} \left( \langle T^{*}Tx, x \rangle^{p/4} - \langle TT^{*}x, x \rangle^{q/4} \right)^{2}.$
\end{lemma}
\begin{proof}
Let $ x \in \mathcal{H}$ be a unit vector. Utilizing (\ref{y1}) and concavity of
$ f(t)=\sqrt{t} $ on $ [0, \infty), $ respectively, we obtain

 \begin{align*}
\Vert Tx \Vert \Vert T^{*}x\Vert &= \langle T^{*}Tx, x \rangle^{1/2}\langle TT^{*}x, x \rangle^{1/2}\\
& \leq \frac{1}{p}\langle T^{*}Tx, x \rangle^{p/2}+ \frac{1}{q}\langle TT^{*}x, x \rangle^{q/2}-
\frac{1}{p} \left( \langle T^{*}Tx, x \rangle^{p/4} - \langle TT^{*}x, x \rangle^{q/4} \right)^{2}\\
& \leq \frac{1}{p}\langle (T^{*}T)^{p}x, x \rangle^{1/2}+ \frac{1}{q}\langle (TT^{*})^{q}x, x \rangle^{1/2}-
\frac{1}{p} \left( \langle T^{*}Tx, x \rangle^{p/4} - \langle TT^{*}x, x \rangle^{q/4} \right)^{2}\\
& \leq \left( \frac{1}{p}\left\langle (T^{*}T)^{p}x, x \right\rangle+ \frac{1}{q}\left\langle (TT^{*})^{q}x, x \right\rangle \right)^{1/2}-
\frac{1}{p} \left( \langle T^{*}Tx, x \rangle^{p/4} - \langle TT^{*}x, x \rangle^{q/4} \right)^{2}\\
& = \left\langle \left( \frac{1}{p}(T^{*}T)^{p}+ \frac{1}{q}(TT^{*})^{q} \right) x, x \right\rangle^{1/2}-
\frac{1}{p} \left( \langle T^{*}Tx, x \rangle^{p/4} - \langle TT^{*}x, x \rangle^{q/4} \right)^{2}\\
& \leq \left\Vert \frac{1}{p}(T^{*}T)^{p}+ \frac{1}{q}(TT^{*})^{q} \right\Vert^{1/2}-
\inf_{\Vert x \Vert=1}\mu(x),\\
\end{align*}
which completes the proof.
\end{proof}

 In line with the proof of Lemma \ref{l7}, we have the following result.
\begin{lemma}\label{l8}
Let $ T \in \mathcal{CR}(\mathcal{H})$ and $ p\geq q >1 $ be conjugate exponents. Then for all
$ x \in \mathcal{H}$ with $ \Vert x \Vert=1, $
$$ \Vert T^{\dagger}Tx \Vert \Vert TT^{\dagger}x\Vert \leq \left\Vert \frac{1}{p} (T^{\dagger}T)^{p}+ \frac{1}{q} (TT^{\dagger})^{q} \right\Vert^{1/2}-\inf_{\Vert x \Vert=1}\mu(x), $$
where $ \mu(x)= \frac{1}{p} \left( \langle T^{\dagger}Tx, x \rangle^{p/4} - \langle TT^{\dagger}x, x \rangle^{q/4} \right)^{2}.$

 \end{lemma}

\begin{theorem}
Let $ T \in \mathcal{CR}(\mathcal{H})$ and $ p\geq q >1, p^{\prime}\geq q^{\prime} > 1$ be, respectively, conjugate exponents. Then
\begin{small}
\begin{align*}
\omega^{2}(T) \leq \left(\left \Vert \frac{1}{p}(T^{*}T)^{p}+ \frac{1}{q}(TT^{*})^{q} \right\Vert^{1/2}-
\inf_{\Vert x \Vert=1}\mu_{p}(x) \right) \left( \left\Vert \frac{1}{p^{\prime}} (T^{\dagger}T)^{p^{\prime}}+ \frac{1}{q^{\prime}} (TT^{\dagger})^{q^{\prime}} \right\Vert^{1/2}-\inf_{\Vert x \Vert=1}\zeta_{p^{\prime}}(x) \right),
\end{align*}
\end{small}
where $ \mu(x)= \frac{1}{p} \left( \langle T^{*}Tx, x \rangle^{p/4} - \langle TT^{*}x, x \rangle^{q/4} \right)^{2}$ and $ \zeta_{p^{\prime}}(x)= \frac{1}{p^{\prime}} \left( \langle T^{\dagger}Tx, x \rangle^{p^{\prime}/4} - \langle TT^{\dagger}x, x \rangle^{q^{\prime}/4} \right)^{2}.$
\end{theorem}
\begin{proof}
Let $ x \in \mathcal{H}$ and $ \Vert x \Vert=1. $ Then Cauchy-Schwarz inequality implies
$$ \vert \langle Tx, x\rangle \vert =\vert \langle Tx, TT^{\dagger}x\rangle \vert \leq \Vert Tx \Vert \Vert TT^{\dagger} \Vert,$$
and
$$ \vert \langle Tx, x\rangle \vert =\vert \langle T^{\dagger}Tx, T^{*}x\rangle \vert \leq \Vert T^{*}x \Vert \Vert T^{\dagger}Tx \Vert.$$
Therefore,
$$ \vert \langle Tx, x\rangle \vert ^{2} \leq \Vert Tx \Vert \Vert T^{*}x \Vert \Vert T^{\dagger}Tx \Vert \Vert TT^{\dagger} x \Vert. $$
Now, by Lemma \ref{l7}, taking the supremum over $ \Vert x \Vert=1, $ we obtain the desired result.
\end{proof}
\begin{remark}
In Lemma \ref{l7}, if $ mI \leq TT^{*} \leq MI$ for some scalars $ 0< m <M, $ then it is easy to see that
$$ \Vert Tx \Vert \Vert T^{*}x\Vert \leq \left\Vert \frac{1}{p} (T^{*}T)^{p/2}+ \frac{1}{qK_{q/2}} (TT^{*})^{q/2} \right\Vert-\inf_{\Vert x \Vert=1}\mu_{p}(x). $$
In the same lemma, if $ m \leq TT^{\dagger} \leq M, $ then
$$ \Vert Tx \Vert \Vert T^{\dagger}x\Vert \leq \left\Vert \frac{1}{p} (T^{\dagger}T)^{p/2}+ \frac{1}{qK_{q/2}} (TT^{\dagger})^{q/2} \right\Vert-\inf_{\Vert x \Vert=1}\zeta_{p}(x), $$

 where $K_{q/2}=K(m, M, q/2)$ and
$$ \mu_{p}(x)= \frac{1}{p} \left( \langle T^{*}Tx, x \rangle^{p/4} - \langle TT^{*}x, x \rangle^{q/4} \right)^{2}, \zeta_{p}(x)= \frac{1}{p} \left( \langle T^{\dagger}Tx, x \rangle^{p/4} - \langle TT^{\dagger}x, x \rangle^{q/4} \right)^{2}.$$

 Now, let $ T \in \mathcal{CR}(\mathcal{H})$ and $ p\geq q >1, \ p^{\prime}\geq q^{\prime} > 1$ be, respectively, conjugate exponents. We have
\begin{enumerate}
\item[(i)] If $ mI \leq TT^{*}, TT^{\dagger} \leq MI$ for some scalars $ 0< m <M, $ then
\begin{small}
\begin{align*}
& \omega^{2}(T) \leq \\
& \left( \left\Vert \frac{1}{p} (T^{*}T)^{p/2}+ \frac{1}{qK_{q/2}} (TT^{*})^{q/2} \right\Vert-\inf_{\Vert x \Vert=1}\mu_{p}(x) \right) \left( \left\Vert \frac{1}{p^{\prime}} (T^{\dagger}T)^{p^{\prime}/2}+ \frac{1}{q^{\prime}K_{q^{\prime}/2}} (TT^{\dagger})^{q^{\prime}/2} \right\Vert-\inf_{\Vert x \Vert=1}\zeta_{p^{\prime}}(x) \right).
\end{align*}
\end{small}
\item[(ii)] If $ mI \leq TT^{*} \leq MI, $ for some scalars $ 0< m <M, $ then
\begin{small}
\begin{align*}
& \omega^{2}(T) \leq \\
& \left( \left\Vert \frac{1}{p} (T^{*}T)^{p/2}+ \frac{1}{qK_{q/2}} (TT^{*})^{q/2} \right\Vert-\inf_{\Vert x \Vert=1}\mu_{p}(x) \right) \left( \left\Vert \frac{1}{p^{\prime}} (T^{\dagger}T)^{p^{\prime}}+ \frac{1}{q^{\prime}} (TT^{\dagger})^{q^{\prime}} \right\Vert^{1/2}-\inf_{\Vert x \Vert=1}\zeta_{p^{\prime}}(x) \right).
\end{align*}
\end{small}
\item[(iii)] If $ mI \leq TT^{\dagger} \leq MI$ for some scalars $ 0< m <M, $ then
\begin{small}
\begin{align*}
& \omega^{2}(T) \leq \\
& \left( \left\Vert \frac{1}{p}(T^{*}T)^{p}+ \frac{1}{q}(TT^{*})^{q} \right\Vert^{1/2}-
\inf_{\Vert x \Vert=1}\mu_{p}(x) \right) \left( \left\Vert \frac{1}{p^{\prime}} (T^{\dagger}T)^{p^{\prime}/2}+ \frac{1}{q^{\prime}K_{q^{\prime}/2}} (TT^{\dagger})^{q^{\prime}/2} \right\Vert-\inf_{\Vert x \Vert=1}\zeta_{p^{\prime}}(x) \right),
\end{align*}
\end{small}
where
\[\mu_{p}(x)= \frac{1}{p} \left( \langle T^{*}Tx, x \rangle^{p/4} - \langle TT^{*}x, x \rangle^{q/4} \right)^{2}, \zeta_{p^{\prime}}(x)= \frac{1}{p^{\prime}} \left( \langle T^{\dagger}Tx, x \rangle^{p^{\prime}/4} - \langle TT^{\dagger}x, x \rangle^{q^{\prime}/4} \right)^{2},\]
and $ K_{\alpha}=K(m, M, \alpha)$.
\end{enumerate}

 \end{remark}


\vspace{0.5cm} \noindent

 {\bf Conflict of Interest}

The authors have no conflicts of interest to declare that are relevant to the content of this article.


\end{document}